\documentclass[12pt]{article}

\usepackage{graphicx} 
\usepackage{amsmath} 
\usepackage{amssymb,amsthm}
\usepackage{color}
\usepackage{tikz}
\usepackage{pgfplots}
\usepackage{enumitem}

\newcommand{\R}{{\mathbb R}}
\newcommand{\C}{{\mathbb C}}
\newcommand{\Res}{\mathrm{Res}}

\newcommand{\cF}{\mathcal{F}}
\newcommand{\cX}{\mathcal{X}}
\newcommand{\cZ}{\mathcal{Z}}

\newcommand{\tLambda}{\widetilde{\Lambda}}
\newcommand{\tR}{\widetilde{R}}
\newcommand{\tf}{\widetilde{f}}

\newcommand{\remo}[1]{\textcolor{gray}{#1}}

\let\epsilon\varepsilon
\let\theta\vartheta

\newtheorem{theorem}{Theorem}[section]\newtheorem{lemma}[theorem]{Lemma}
\newtheorem{definition}[theorem]{Definition}

\newtheorem{remark}[theorem]{Remark}

\title{A Ginzburg-Landau approximation theorem for quasilinear pattern-forming systems \\ in uniformly local Sobolev spaces}
\author{Theo Belin$^1$, Guido Schneider$^2$ \\
{\small
$^1$
Centre for Mathematical Sciences, Lund University,} \\ {\small  S\"olvegatan 18A,
223 62 Lund,
Sweden} \\
{\small
$^2$
Institut f\"ur Analysis, Dynamik und Modellierung, Universit\"at Stuttgart, } \\ {\small  Pfaffenwaldring 57,  70569 Stuttgart, Germany}}

\begin{document}

\maketitle

\begin{abstract}
We prove that the Ginzburg-Landau equation correctly predicts the dynamics of quasilinear 
or fully nonlinear
pattern-forming systems close to the 
first instability.
We present an approximation theory in uniform local Sobolev spaces 
which is applied to a quasilinear
Swift-Hohenberg model, to a quasilinear version of B\'enard's problem with velocity dependent 
viscosity, and to abstract quasilinear reaction-diffusion-advection systems.
\end{abstract}




%

%
%
%
%
%
%
%
%
%

\section{Introduction}

In the following we consider quasilinear 
pattern-forming systems such as a quasilinear
Swift-Hohenberg model, a quasilinear version of B\'enard's problem, 
or abstract quasilinear reaction-diffusion-advection systems. 

We are interested in describing these systems near the first instability 
in a parameter regime where a Turing instability occurs.
The real Ginzburg-Landau equation 
$$ 
\partial_T A =  \nu_2 \partial_X^2 A + \nu_1 A + \nu_3 A |A|^2 ,
$$
with coefficients $ \nu_2 > 0 $, $ \nu_1,\nu_3 \in \R $, time variable 
$ T \geq 0 $, space variable $ X \in \R $, and solution $ A(X,T) \in \C $,
is a universal amplitude equation that can be derived by a multiple scaling perturbation ansatz and can be used to approximately describe such a situation.
Error estimates were established in several papers, starting with \cite{CE90,vH91,KSM92,Schn94a,Schn94ZAMP}, which show that the Ginzburg-Landau equation makes correct predictions about the dynamics of the original pattern-forming system.
Among other things,
these error estimates have been used to prove the global existence of solutions starting in a small neighborhood of the weakly unstable origin, cf. \cite{MS95,Schn99JMPA}.
The textbook \cite{SU17book} provides an overview and introduction to the theory.

However, for quasilinear systems only two 
approximation results are known to us. In \cite{Zi14PhD}  
the Ginzburg-Landau approximation 
was justified for the Marangoni problem in Sobolev spaces $ H^r $.
In \cite{BS26} an easy-to-use approximation result for quasilinear pattern-forming reaction-diffusion-advection systems was presented which was applied to the Gray-Scott-Klausmeier system.

Since in general the bifurcating solutions do not vanish
for $ |x| \to \infty $, it is essential to prove such approximation
results also in function spaces, such as 
$H^r_{l,u}$,  $ C^r_{b,unif}$, or $ C^{r,\alpha} $.
Such approximation
results
are well-established for 
semilinear problems but for quasilinear systems such a result was not 
proven so far.
Therefore, it is the purpose of this paper to prove such an approximation result in the quasilinear case, too. 

The examples which we choose to present our approach are adaptions 
of well-studied semilinear examples such that we can refer to  existing 
literature for the derivation of the Ginzburg-Landau equation and for the 
estimates of the residual terms. This allows us to concentrate on the handling 
of the quasilinear feature.

The plan of the paper is as follows. In Section \ref{sec2} we 
introduce a quasilinear Swift-Hohenberg model for which we explain the 
underlying ideas. 
The main technical part is the optimal regularity estimates for the solutions
of the inhomogeneous linearized problem which are carried out 
in Section \ref{sec2b} and which are based on $C^{\alpha}$-theory in time.
In Section \ref{sec3} we improve these ideas to handle a
quasilinear version of B\'enard's problem. 
Finally, in Section \ref{sec4}
we present an abstract Ginzburg-Landau approximation result for  
quasilinear pattern-forming reaction-diffusion-advection systems.
The paper is closed with a discussion and outlook section.
\medskip

{\bf Acknowledgement:} This work  was partially supported by the 
DFG Network program 'Instability Phenomena in Asymptotic Models in Fluid Dynamics'
which is funded by the Germany Research Foundation (DFG) project number 545145736.

\section{The quasilinear Swift-Hohenberg model}
\label{sec2}

The Swift-Hohenberg equation 
$$
\partial_t u = - (1+\partial_x^2)^2 u + \alpha_{SH} u - u^3 
$$
with $ t \geq 0 $, $ x \in \R $, $ u(x,t) \in \R $, parameter $ \alpha_{SH} \in \R $
was used as a prototype model in the justification analysis of the 
Ginzburg-Landau approximation \cite{CE90,KSM92}.
Here, we are interested in the following fully nonlinear version 
\begin{equation} \label{qSH}
\partial_t u = - (1+\partial_x^2)^2 u + \alpha_{SH} u  + 
g(u, \ldots,\partial_x^4 u) ,
\end{equation}
with a smooth function $ g:\R^5 \to \R $ without constant and linear terms.
%
%
For the following explanations and our purposes it is sufficient to consider
$$ 
g(u, \ldots,\partial_x^4)  =  - \partial_x^4 (u^3).
$$ 
The linearization around the trivial solution $ u = 0 $ is solved by 
$ u(x,t) = e^{ikx + \lambda(k,\alpha) t} $ where 
$$ 
\lambda(k,\alpha_{SH}) = -(1-k^2)^2 + \alpha_{SH}.
$$ 
Hence, $ u = 0 $ becomes unstable for $ \alpha_{SH} = 0 $ at the 
wave numbers $ k = \pm 1 $. 
For $ \alpha_{SH} > 0 $ we introduce the small perturbation parameter $ 0 < \varepsilon^2 \ll 1 $
by $ \alpha_{SH} = \varepsilon^2 $
to describe the bifurcating solutions. For the derivation of the Ginzburg-Landau equation 
we make the ansatz 
\begin{equation} \label{ansatz}
u(x,t) = \varepsilon \Psi_{GL}(x,t) = \varepsilon A(\varepsilon x,\varepsilon^2 t) e^{ix} + c.c.,
\end{equation}
with slow spatial scale $ X = \varepsilon x $, slow time scale 
$ \varepsilon^2 t $, and amplitude $ A(X,T) \in \C $, modulating the 
weakly unstable modes $ e^{\pm i x} $.
Inserting this ansatz in \eqref{qSH} and equating the coefficients in front 
of the $ \varepsilon^{j_1} e^{i j_2 x} $ to zero gives an equation 
for the amplitude function $ A $.
The coefficients in front of $ \varepsilon e^{ix} $ and $ \varepsilon^2 e^{ix} $
vanish identically. At $ \varepsilon^3 e^{ix} $ we find the 
Ginzburg-Landau equation 
\begin{equation} \label{GL}
\partial_T A = 4 \partial_X^2 A + A - 3 A |A|^2 .
\end{equation}
In this paper we are interested in estimating the error made by the
approximation $ \varepsilon \Psi_{GL} $. 
For the reasons explained in the introduction
we will do so in $ H^r_{l,u} $-spaces 
equipped with the norm
$$ 
\| u \|_{H^r_{l,u}} = \sup_{y \in \R} \| u \|_{H^r(y, y+1)}  = \sup_{y \in \R}
\left(\sum_{j= 0}^r \int_y^{y+1} |\partial_x^j u(x)|^2 dx\right)^{1/2}.
$$ 
For a precise definition of $ H^r_{l,u} $-spaces we refer to \cite[\S 8.3.1]{SU17book}.
For the quasilinear Swift-Hohenberg model \eqref{qSH} we prove:
\begin{theorem}
Let $ A \in C([0,T_0],H^{r_0+2}_{l,u}) $ be a solution of the 
Ginzburg-Landau equation \eqref{GL} satisfying 
\begin{equation} \label{eq19s2}
\sup_{T \in [0,T_0]} \| A(\cdot,T) \|_{H^{r_0+2}_{l,u}}
\leq C_{GL}.
\end{equation}
Then there exist $ \varepsilon_0 > 0 $ 
and $ C_2 > 0 $, only depending on 
$ C_{GL} > 0 $, $ T_0 > 0 $, $ r_0 \geq 5 $,  such that  for all
$ \varepsilon \in (0,\varepsilon_0) $ there are solutions 
$ u $ of the quasilinear Swift-Hohenberg model \eqref{qSH}
with
$$ 
\sup_{t \in [0,T_0/\varepsilon^2]} \| u(\cdot,t) - \varepsilon \Psi_{GL}(\cdot,t) \|_{H^{r_0}_{l,u}} \leq C_2 \varepsilon^{3/2}.
$$
\end{theorem}
\begin{remark}{\rm
Sobolev's embedding theorem $ H^{r_0}_{l,u}(\R,\R) \subset C^0_{b,unif}(\R,\R)  $ for $ r_0 \geq 1 $ gives 
the estimate
$$ 
\sup_{t \in [0,T_0/\varepsilon^2]} 
\sup_{x \in \R}
| u(x,t) -  \varepsilon \Psi_{GL}(x,t) | \leq C_2 \varepsilon^{3/2}.
$$}
\end{remark}
\noindent
{\bf Proof.}
In order to use the subsequent proofs also in more general situations we 
introduce the following notation.
We set 
$$ 
\cX = H^{r}_{l,u} \qquad \textrm{and} \qquad D(\Lambda) = H^{r+4}_{l,u} ,
$$ 
where $ r_0= r+ 4 $.
In a first step we have to estimate the 
residual $ \Res(\varepsilon \Psi_{GL}) $, with
$$ 
\Res(u) = - \partial_t u  - (1+\partial_x^2)^2 u + \varepsilon^2  u  - \partial_x^4 (u^3),
$$ 
i.e., the terms which do not cancel after inserting the approximation $ \varepsilon \Psi_{GL} $
into the quasilinear Swift-Hohenberg model \eqref{qSH}.
For estimating the error we need the residual $ \Res(\varepsilon \Psi_{GL}) $
to be sufficiently small, namely $  \mathcal{O}(\varepsilon^{7/2})$ in $ \cX $.
To achieve this bound  we follow the existing literature \cite{KSM92} and use the improved approximation
$$
\varepsilon \Psi = \varepsilon \Psi_{GL} + \varepsilon^3 A_3(\varepsilon x,\varepsilon^2 t) e^{3 ix} + c.c.,
$$
where $ A_3 $ satisfies 
$$ 
0 = - 64 A_3 - 81 A^3 .
$$
For the residual of the improved approximation we find
\begin{lemma} \label{lem51}
Let $ A \in C([0,T_0],H^{r+4}_{l,u}) $ be a solution of the 
Ginzburg-Landau equation \eqref{GL} satisfying \eqref{eq19s2}.
Then there exist $ \varepsilon_0 > 0 $ 
and $ C_{res} > 0 $, only depending on $ C_{GL} > 0 $, $ T_0 > 0 $, and $ r \geq 1 $, such that   for all $ \varepsilon \in (0,\varepsilon_0)
$ we have 
$$
\sup_{t \in [0,T_0/\varepsilon^2]} \| \Res(\varepsilon \Psi) \|_{\cX} \leq C_{res} \varepsilon^{7/2} ,
$$ 
where $ \cX = H^{r}_{l,u} $.
\end{lemma}
%
%
\noindent
{\bf Proof.}
The proof follows almost line for line the one given in \cite[Section 10.2]{SU17book} for 
the classical Swift-Hohenberg equation.
\qed 
\medskip

For estimating the error made by this formal approximation we write the solution $ u $ as a sum of the improved approximation $ \varepsilon \Psi $ 
and an error $  \varepsilon^{\beta} R $ with $ \beta > 1 $ chosen below.
The error satisfies 
\begin{equation} \label{linearSH}
\partial_t R = \Lambda R + f = - (1+\partial_x^2)^2 R +  \varepsilon^2 R+ f ,
\end{equation}
with 
\begin{eqnarray*} 
f &= &   - 3 \varepsilon^2  \partial_x^4(\Psi^2 R) -  3 \varepsilon^{1+\beta} \partial_x^4(\Psi R^2)
- \varepsilon^{2 \beta}   \partial_x^4 (R^3)  + \varepsilon^{-\beta} \Res(\varepsilon \Psi).
\end{eqnarray*}
There exist $ C_1 $, $ C_2 $, and $ C_3 $ such that for all $ \varepsilon \in (0,1) $ we have 
\begin{equation} \label{festi}
\| f \|_{\cX} \leq C_1 \varepsilon^2 \| R \|_{D(\Lambda)}
+ C_2 \varepsilon^{1+\beta}  \| R \|_{D(\Lambda)}^2 + C_3 \varepsilon^{2 \beta}  
\| R \|_{D(\Lambda)}^3 + C_{res} \varepsilon^{7/2-\beta}.
\end{equation}
Since we have to prove estimates on the long $ [0,T_0/\varepsilon^2] $-time scale 
all terms of $ f $ need at least an $ \varepsilon^2 $ in front. Therefore, we set $ \beta = 3/2 $ in the following.

To handle the quasilinear character of the equations we introduce two families of spaces.
For $\alpha \in (0, 1)$ define
$$
X^{r,\alpha}_{\eta,t_0} =  C^{\alpha}([0,t_0],\cX)
$$
equipped with the norm
$$
\| R \|_{X^{r,\alpha}_{\eta,t_0}} = \sup_{t \in [0,t_0]} \| e^{- \eta t} R(t) \|_{C^{\alpha}([0,t_0],\cX)}
$$
and 
$$
Y^{r+4,\alpha}_{\eta,t_0} = C^{\alpha}([0,t_0],D(\Lambda)) \cap C^{1+\alpha}([0,t_0],\cX)
$$
equipped with the norm
$$
\| R \|_{Y^{r+4,\alpha}_{\eta,t_0}} = \sup_{t \in [0,t_0]} \| e^{- \eta t} R(t) \|_{C^{\alpha}([0,t_0],D(\Lambda))} + \sup_{t \in [0,t_0]} \| e^{- \eta t} R(t) \|_{C^{1,\alpha}([0,t_0],\cX)}.
$$
Since for fixed $ t_0 \in (0,\infty) $ these norms are equivalent 
to the unweighted norms, without the factor $e^{- \eta t} $,
the terms collected in $ f $ will form a smooth mapping from $  Y^{r+4,\alpha}_{\eta,t_0} $ 
to $ X^{r,\alpha}_{\eta,t_0} $. Using optimal regularity will show that the solutions 
of  \eqref{linearSH} define a mapping $ \mathcal{K}: f \mapsto R $ which maps 
 $ X^{r,\alpha}_{\eta,t_0} $ to $ Y^{r+4,\alpha}_{\eta,t_0} $.

Since we have to bound the solutions of the error equations on the long $ [0,T_0/\varepsilon^2] $-time interval the spaces $ X^{r,\alpha}_{\eta,t_0} $ and $ Y^{r+4,\alpha}_{\eta,t_0} $
are equipped with exponential weights in time, in the following we choose $ \eta = \widetilde{\eta} \varepsilon^2 $
with $ \widetilde{\eta} > 0 $ fixed, sufficiently large and independent of $ 0 < \varepsilon^2 \ll 1 $ and $t_0 = T_0/\varepsilon^2$. 
Again, to use the subsequent proofs also in more general situations we 
introduce the following notation.
We set 
$$ 
\cZ^0 = X^{r,\alpha}_{\widetilde{\eta} \varepsilon^2,T_0/\varepsilon^2}  \qquad \textrm{and} \qquad \cZ^1 = Y^{r+4,\alpha}_{\widetilde{\eta} \varepsilon^2,T_0/\varepsilon^2}  .
$$ 

The solution of the error equation \eqref{linearSH} can be constructed through a fixed point argument. The right-hand side of $ R = \mathcal{K} f(R) $ will be a contraction in $\cZ^1$ for $ \varepsilon > 0 $ sufficiently small and $ \widetilde{\eta} > 0 $ sufficiently large, independent of $ 0 < \varepsilon \ll 1 $.
 Since $ e^{\widetilde{\eta} \varepsilon^2 t} = \mathcal{O}(1) $ for 
 $ t \in [0,T_0/\varepsilon^2]$ all estimates transfer 
 one-to-one from $ H^{r+4}_{l,u} $-spaces to $ Y^{r+4,\alpha}_{\widetilde{\eta} \varepsilon^2,T_0/\varepsilon^2} $-spaces. In detail, 
using 
the subsequent Remark \ref{remmakre} and 
the fact that $ \varepsilon \Psi $ is uniformly bounded for
$ t \in [0,T_0/\varepsilon^2] $,
 \eqref{festi} transfers into:
 
 There exist $ C_1 $, $ C_2 $, and $ C_3 $ such that for all $ \varepsilon \in (0,1) $ we have 
\begin{eqnarray} \label{festi1}
\| f \|_{\cZ^0} & \leq & C_1 \varepsilon^2 \| R \|_{\cZ^1}
+ C_2 \varepsilon^{1+\beta} e^{\widetilde{\eta} T_0} \| R \|_{\cZ^1}^2 \\ && \qquad + C_3 \varepsilon^{2 \beta}  e^{2 \widetilde{\eta} T_0}
\| R \|_{\cZ^1}^3 + C_{res} \varepsilon^{7/2-\beta}. \nonumber
\end{eqnarray}
\begin{remark}\label{remmakre}{\rm
Since we need to estimate the residual in $ \cZ^0 =  C^{\alpha}([0, T_0/\epsilon^2], \cX)$ we need  more regularity in time compared to the  estimates for the residual in Lemma \ref{lem51}.
\begin{lemma} \label{lem51b}
Let  $ \alpha \in (0,1) $ and let $ A \in C([0,T_0],H^{r+4+2 \alpha}_{l,u}) $ 
be a solution of the 
Ginzburg-Landau equation \eqref{GL} satisfying \eqref{eq19s2}.
Then there exist $ \varepsilon_0 > 0 $ 
and $ C_{res} > 0 $, only depending on $ C_{GL} > 0 $, $ T_0 > 0 $, and $ r \geq 1 $, such that for all $ \varepsilon \in (0,\varepsilon_0)
$  we have 
$$
 \| \Res(\varepsilon \Psi) \|_{\cZ^0} \leq C_{res} \varepsilon^{7/2},
$$ 
where $ \cZ^0 = X^{r,\alpha}_{\widetilde{\eta} \varepsilon^2,T_0/\varepsilon^2} $. 
\end{lemma}
}\end{remark}
\noindent
{\bf Proof.} Observe that the assumption on $ A $ imply that $A \in C^{\alpha}([0, T_0], H^{r + 4}_{l, u})$, thus, the proof follows almost line for line the one given in \cite[Section 10.2]{SU17book} for 
the classical Swift-Hohenberg equation. 
\qed
\medskip

Note that $\Lambda$ generates an analytic semigroup on $\cX$. There exists $C > 0$ and $h > 0$, independent of $0 < \varepsilon \ll 1$, such that 
\begin{equation}
  \label{sg_estimate}
  \|e^{t\Lambda}\|_{\cX \to \cX} \leq C e^{h\varepsilon^2 t},
\end{equation}
cf.  \cite[Lemma 10.2.8]{SU17book}.
For the linear terms we use the following optimal regularity lemma which is an adaptation of 
\cite[Theorem 4.3.1]{Lunardibook1}. For the formulation 
of the lemma
we need an additional space.
\begin{definition}
For $ 0 < \alpha < 1 $ the space $ D_{\Lambda} (\alpha,\infty) $ is 
defined in terms of the semigroup $ e^{t \Lambda} $, as the set of all $ u \in \cX $
such that $ t^{1-\alpha} \| \Lambda e^{t \Lambda} u \|_{\cX} $
is bounded near $ t = 0 $. It will be equipped with the norm 
$$ 
\| u \|_{D_{\Lambda} (\alpha,\infty)} = \sup_{t \in [0,1]} t^{1-\alpha} \| \Lambda e^{t \Lambda} u \|_{\cX} .
$$
\end{definition}
\begin{lemma} \label{lem23}
For all $ \alpha \in (0,1) $, $ r \geq 0 $, $ T_0 > 0 $,  there exists a   
$ C > 0 $ such that for all $ \varepsilon \in (0,1) $ the following holds.
Let
$ f \in \cZ^0 $,
$ R_0  \in D(\Lambda)$, and
let $ R $ be the mild solution of  
\begin{equation} \label{eq401}
\partial_t R = \Lambda R + f  = - (1+\partial_x^2)^2 R + \epsilon^2 R + f 
\end{equation}
with $ R|_{t = 0} = R_0 $.
Then we have that the solution $ R = \mathcal{K} f \in \cZ^1 $ satisfies
\begin{eqnarray*}
\| R \|_{\cZ^1} 
 & \leq & C ( (\widetilde{\eta}  - h)^{-1}\varepsilon^{-2} + 1) 
  (\| f \|_{\cZ^0} + \| R_0 \|_{D(\Lambda)}+
\|\Lambda R_0 + f |_{t=0}\|_{D_{\Lambda} (\alpha,\infty)})
\end{eqnarray*}
for all $ \widetilde{\eta} >  h $.
\end{lemma}
\noindent
\textbf{Proof.} 
See Section \ref{sec2b}.
\qed
\medskip


 The error estimates will follow with a fixed point argument. From 
\eqref{linearSH} we obtain
\begin{eqnarray} \label{erreq1inva}
R & = & \mathcal{K}( f(R)).
\end{eqnarray}
We prove that the mapping $ R \mapsto \mathcal{K}( f(R)) $
is a contraction in a ball in the space 
$
\mathcal{Z}^1 $. 
We choose $ R_0 = 0 $
so that 
$$ 
\Lambda R_0 + f |_{t=0} = \varepsilon^{-\beta} \Res(\varepsilon \Psi)|_{t=0}.
$$
Therefore, by standard analytic semigroup theory combined with
the multiplier lemma \cite[Lemma 8.3.7]{SU17book} in $ H^r_{l,u}$-spaces, the term 
\begin{eqnarray*}
\lefteqn{\|\varepsilon^{-\beta} \Res(\varepsilon \Psi)|_{t = 0}\|_{D_{\Lambda} (\alpha,\infty)}}
\\ & = & \sup_{t \in [0,1]}
t^{1-\alpha} \| \Lambda e^{t \Lambda} (\varepsilon^{-\beta} \Res(\varepsilon \Psi)|_{t = 0}) \|_{H^r_{l,u}}
\\ & \leq & \sup_{t \in [0,1]} t^{1-\alpha} \| \Lambda e^{t \Lambda} \|_{H^{r+4 \alpha}_{l, u} \to H^{r}_{l, u}} \|\varepsilon^{-\beta} \Res(\varepsilon \Psi)|_{t = 0}) \|_{H^{r+4 \alpha}_{l,u}} 
\\ & \leq & C \|\varepsilon^{-\beta} \Res(\varepsilon \Psi)|_{t = 0}) \|_{H^{r+4 \alpha}_{l,u}} 
\end{eqnarray*}
is $ \mathcal{O}(1) $-bounded w.r.t. $ \varepsilon $ for $ t \to 0 $ if 
$ \Res(\varepsilon \Psi)|_{t = 0} \in H^{r+ 4 \alpha}_{l,u} $ is $ \mathcal{O}(\varepsilon^{7/2}) $-bounded 
w.r.t. $ \varepsilon $.


Hence, we can estimate
\begin{eqnarray*}  
 \|\mathcal{K}( f(R))  \|_{\mathcal{Z}^1} 
 & \leq & C( \varepsilon^2 + (\widetilde{\eta} - h)^{-1}) \left(C_1 \| R \|_{\mathcal{Z}^1}  + C_2 \varepsilon^{\beta-1} e^{\widetilde{\eta} T_0} \| R \|_{\mathcal{Z}^1}^2 \right.\\
 && \left. + C_3 \varepsilon^{2(\beta-1)} e^{2\widetilde{\eta} T_0}\| R \|_{\mathcal{Z}^1}^3
+ C_{res}\right).
\end{eqnarray*}
Therefore, the right hand side of \eqref{erreq1inva} maps a ball of 
$ \mathcal{Z}^1 $ with fixed, but sufficiently large 
radius $ \rho_0 > 0 $ in itself when $ \widetilde{\eta}  > h$ is chosen sufficiently large to control the first term and last term, and then finally $ \varepsilon > 0 $ is chosen sufficiently 
small to control the first, second and third term.
With the same argument the right hand side of \eqref{erreq1inva} can be shown to be  a contraction in this ball of radius $ \rho_0 $.
Hence, there exists a unique fixed point in this ball for this mapping.  \qed

\section{Inverting the linearization}
\label{sec2b}

This section contains the proof of Lemma \ref{lem23}.
To estimate the solutions of 
\begin{equation}  \label{eq401ta}
\partial_t R = \Lambda R + f  ,
\end{equation}
we separate the critical modes from the stable ones. To accomplish this, 
we introduce some mode filters $ E_{c,\pm 1} = \cF^{-1} \widehat{E}_{c,\pm 1} \cF$ 
by the multiplication operators $ \widehat{E}_{c,\pm 1} \in C_0^{\infty} $
defined by
$$ 
\widehat{E}_{c,\pm 1}(k) = \left\{ \begin{array}{cl} 1, & k   \in  [\pm k_c- k_c/20,\pm k_c + k_c/20] , \\
0 ,&  k \not   \in  (\pm k_c- k_c/10,\pm k_c + k_c/10) , \\
\in [0,1] , & \textrm{elsewhere},
\end{array} \right.
$$
in Fourier space. 
Moreover, we define 
$$ 
E_{\sigma} = I - E_{c,1} - E_{c,-1}.
$$ 
We apply these mode filters to \eqref{eq401ta} and find
\begin{eqnarray}  \label{eq401t1a}
\partial_t R_{c,\pm 1} & = & \Lambda R_{c,\pm 1} + E_{c,\pm 1} f , \\ 
\partial_t R_{\sigma} & = & \Lambda R_{\sigma} + E_{\sigma} f .
 \label{eq401ts}
\end{eqnarray}
The  estimates for the $ R_{c,\pm 1} $-parts follow by classical 
semigroup theory for semilinear systems due to the 
compact support in Fourier space of the multipliers $ \widehat{E}_{c,\pm 1} $.
We obtain
\begin{lemma}
\label{lemestimateEcpart}
For all  $ \alpha \in (0,1) $, $ r \geq 0 $, $ T_0 > 0 $,  there exists a   
$ C > 0 $ such that for all $ \varepsilon \in (0,1) $ the following holds.
Let
$ f \in \cZ^0 $,
$ R_0  \in D(\Lambda) $, and
let $ R_{c,\pm 1}  $ be the mild solution of  
\begin{equation} \label{eq401}
\partial_t R_{c,\pm 1} = \Lambda R_{c,\pm 1} + E_{c,\pm 1} f  = - (1+\partial_x^2)^2 R_{c,\pm 1} + \varepsilon^2 R_{c,\pm 1}  +  {E}_{c,\pm 1} f 
\end{equation}
with $ R_{c,\pm 1} |_{t = 0} = {E}_{c,\pm 1} R_0 $.
Then we have that the solution $ R_{c,\pm 1} = \mathcal{K} E_{c,\pm 1} f \in \cZ^1$ satisfies
\begin{eqnarray*}
\| R_{c,\pm 1}  \|_{\cZ^1} 
 & \leq & C ( (\widetilde{\eta} - h)^{-1} \varepsilon^{-2} + 1)
 (\| E_{c,\pm 1} f \|_{\cZ^0} + \| R_0 \|_{D(\Lambda)})
\end{eqnarray*}
for all $ \widetilde{\eta} > h $.
\end{lemma}
The proof will be given below. 
For the $ R_{\sigma} $-part we use optimal regularity estimates and the fact 
that linear semigroup is exponentially damped. 
\begin{lemma}
\label{lemestimateErpart}
For all  $ \alpha \in (0,1) $, $ r \geq 0 $, $ T_0 > 0 $, there exists a   
$ C > 0 $ such that for all $ \varepsilon \in (0,1) $ the following holds.
Let
$ f \in \cZ^0 $,
$ R_0  \in D(\Lambda) $, and
let $ R_{\sigma} $ be the mild solution of  
\begin{equation} \label{eq401}
\partial_t R_{\sigma} = \Lambda R_{\sigma} + E_{\sigma} f  = - (1+\partial_x^2)^2 R_{\sigma} + \varepsilon^2 R_{\sigma}  + E_{\sigma} f 
\end{equation}
with $ R_{\sigma}|_{t = 0} = E_{\sigma} R_0 $.
Then we have that the solution $ R_{\sigma} = \mathcal{K} E_{\sigma} f \in \cZ^1 $ satisfies
\begin{eqnarray*}
\| R_{\sigma} \|_{\cZ^1} 
 & \leq & C  
 (\| E_{\sigma}f \|_{\cZ^0} + \| R_0 \|_{D(\Lambda)}+
\|\Lambda R_0 + f |_{t=0}\|_{D_{\Lambda} (\alpha,\infty)})
\end{eqnarray*}
for all $ \widetilde{\eta}  > 0 $.
\end{lemma}
Obviously Lemma \ref{lemestimateEcpart} and Lemma \ref{lemestimateErpart} 
imply the validity of Lemma \ref{lem23}.
\medskip

Before we give the proofs of Lemma \ref{lemestimateEcpart} and Lemma \ref{lemestimateErpart} we remark that in case the nonlinearity contains quadratic terms Lemma \ref{lem23} no longer can be used.
In this case we have to work directly with  
Lemma \ref{lemestimateEcpart} and Lemma \ref{lemestimateErpart}.
See the subsequent sections.
\medskip

In the proofs of Lemma \ref{lemestimateEcpart} and Lemma \ref{lemestimateErpart} 
we use the following notation.
We introduce  $ \tR(t) = R(t)  e^{- \eta t} $
and find 
\begin{equation}  \label{eq401t}
\partial_t \tR = \tLambda \tR + \tf  
\end{equation}
with 
\begin{eqnarray*}
\tLambda & = & \Lambda - \eta I ,\\
\tf(\cdot,t) & = & f(\cdot,t) e^{- \eta t}.
\end{eqnarray*}
Applying the mode filters to \eqref{eq401t} and find
\begin{eqnarray}  \label{eq401t1}
\partial_t \tR_{c,\pm 1} & = & \tLambda \tR_{c,\pm 1} + E_{c,\pm 1} \tf , \\ 
\partial_t \tR_{\sigma} & = & \tLambda \tR_{\sigma} + E_{\sigma} \tf .
 \label{eq401ts}
\end{eqnarray}

We start with the more involved 

{\bf Proof of Lemma \ref{lemestimateErpart}.}
We recall the abbreviations $ \cX = H^r_{l,u} $ and 
$ D(\Lambda) = H^{r+\mu}_{l,u} $ with $ \mu = 4 $. 
Moreover we set $ \cX^{\rho} = H^{r+ \mu \rho}_{l,u} $.

For the subsequent estimates we use that the semigroup restricted 
to the $ E_{\sigma} $-part is damped with an exponential rate $ - \beta_- $ 
for a $  \beta_-  > 0 $ independent of the small perturbation parameter 
$ 0 < \varepsilon^2 \ll 1 $, i.e.
\begin{equation} \label{bolognadecay}
\| e^{t \tLambda}  \tR_{\sigma}  \|_{\cX^{\rho}} \leq C t^{-\rho} e^{- \beta_- t} \|  \tR_{\sigma}  \|_{\cX}
\end{equation}
for all $ r, \rho \geq 0 $.
First of all, we have to estimate the quantities 
\begin{equation} \label{bologna}
\sup_{t \in [0,T_0/\varepsilon^2]} \|\tR_{\sigma}(\cdot,t) \|_{\cX} 
\qquad 
\textrm{and}
\qquad 
\sup_{t \in [0,T_0/\varepsilon^2]} \|\tR_{\sigma}(\cdot,t) \|_{D(\Lambda)}.
\end{equation}
Moreover, we have to estimate   the quantities 
\begin{equation} \label{bologna2a}
\|\partial_t(\tR_{\sigma}(\cdot,t)- \tR_{\sigma}(\cdot,s)) \|_{\cX}  
\end{equation}
and
\begin{equation} \label{bologna2b}
\|\tR_{\sigma}(\cdot,t)- \tR_{\sigma}(\cdot,s) \|_{D(\Lambda)} 
\end{equation}
in terms of 
$  \| f \|_{C^{\alpha}(\cX)}  (t-s)^{\alpha} $. By \eqref{eq401ts} the estimate for \eqref{bologna2a}
will be a direct consequence of the estimate for \eqref{bologna2b}.
Hence, it is sufficient to estimate \eqref{bologna} and \eqref{bologna2b}.

For estimating the quantities  \eqref{bologna} we consider the variation of constant formula 
$$ 
\tR_{\sigma}(t) = e^{t \tLambda} R_0 + \int_0^t e^{(t-s) \tLambda} E_{\sigma} \tf(s) ds.
$$
With  \eqref{bolognadecay} we obtain 
\begin{eqnarray*}
\lefteqn{
\sup_{t \in [0,T_0/\varepsilon^2]} \|\tR_{\sigma}(\cdot,t) \|_{\cX} } \\  & \leq & 
C   \|\tR_{\sigma}(\cdot,0) \|_{\cX}  
+ \sup_{t \in [0,T_0/\varepsilon^2]} C \int_0^t e^{-\beta_-(t-s) } \|\tf(s)\|_{\cX}  ds 
\\  & \leq & C  \|\tR_{\sigma}(\cdot,0) \|_{\cX}   +  \int_0^{T_0/\varepsilon^2} e^{-\beta_-(t-s) } ds 
\sup_{t \in [0,T_0/\varepsilon^2]} \|\tf(t)\|_{\cX} 
\\  & \leq & C  \|\tR_{\sigma}(\cdot,0) \|_{\cX}   + \frac{C}{\beta_-} 
\sup_{t \in [0,T_0/\varepsilon^2]} \|\tf(t)\|_{\cX} .
\end{eqnarray*}
We follow \cite[Lemma 3.2.1]{He81} and write 
$$ 
\tLambda \tR_{\sigma}(t) = \tLambda e^{t \tLambda} R_0 + \int_0^t \tLambda e^{(t-s) \tLambda} E_{\sigma} (\tf(s) - \tf(t)) ds + (I- e^{t \tLambda})E_{\sigma}  \tf(t) .
$$
We obtain 
\begin{eqnarray*}
\lefteqn{
\sup_{t \in [0,T_0/\varepsilon^2]} \|\tLambda \tR_{\sigma}(\cdot,t) \|_{\cX} } \\  & \leq & 
C   \|\tLambda \tR_{\sigma}(\cdot,0) \|_{\cX}   + 
 \sup_{t \in [0,T_0/\varepsilon^2]} C \int_0^t e^{-\beta_-(t-s) } (t-s)^{\alpha-1} ds \| f \|_{C^{\alpha}(\cX)} 
 \\ && 
 + \sup_{t \in [0,T_0/\varepsilon^2]} \| (I- e^{t \tLambda}) E_{\sigma} \tf(t) \|_{\cX} 
 \\  & \leq & 
C   \|\tLambda \tR_{\sigma}(\cdot,0) \|_{\cX}   + 
C \int_0^{T_0/\varepsilon^2} e^{-\beta_-(t-s) } (t-s)^{\alpha-1} ds \| f \|_{C^{\alpha}(\cX)} 
 \\ && 
 + C \sup_{t \in [0,T_0/\varepsilon^2]} \|  E_{\sigma} \tf(t) \|_{\cX} 
  \\  & \leq & C   \|\tLambda \tR_{\sigma}(\cdot,0) \|_{\cX}  + C(\alpha,\beta_-) \| f \|_{C^{\alpha}(\cX)} ,
\end{eqnarray*}
with $ C(\alpha,\beta_-) = \mathcal{O}(1) $ for $ \varepsilon \to 0 $. 
\medskip

We come to the remaining estimate, namely that  
of the  quantity  \eqref{bologna2b}. This part is an adaptation of 
\cite[Theorem 4.3.1]{Lunardibook1}.
Let $ \tR_{\sigma} = R_1 + R_2 $ be the mild solution of 
\eqref{eq401ts}, where $ R_1 $ and $ R_2 $ satisfy
\begin{eqnarray*}
R_1(t) & = & \int_0^t e^{(t-s) \tLambda} E_{\sigma} (\tf(s)-\tf(t)) ds, \\
R_2(t) & = & e^{t \tLambda} R_0 + \int_0^t e^{(t-s) \tLambda} E_{\sigma} \tf(t) ds.
\end{eqnarray*}
By applying $ \tLambda $ on these expressions, followed by 
an
explicit integration in the second line,  we find that
\begin{eqnarray*}
\tLambda R_1(t) & = & \int_0^t \tLambda e^{(t-s) \tLambda} E_{\sigma} (\tf(s)-\tf(t)) ds, \\
\tLambda R_2(t) & = & \tLambda e^{t \tLambda} R_0 +  (e^{t \tLambda}-1) E_{\sigma} \tf(t) .
\end{eqnarray*}
With some abuse of notation we write $ f $ instead of $ E_{\sigma} \tf $ in the following. 
\medskip

{\bf a)} To show that $ \tLambda R_1 $ is H\"older continuous in  $ [0,t_0] $ 
we consider 
\begin{eqnarray*}
\lefteqn{\tLambda R_1(t) - \tLambda R_1(s)}  \\
& = & \int_0^t \tLambda e^{(t-\sigma) \tLambda} (f(\sigma)-f(t)) d\sigma
- \int_0^s \tLambda e^{(s-\sigma) \tLambda} (f(\sigma)-f(s)) d\sigma
\\
& = & s_1 + s_2 + s_3 ,
\end{eqnarray*}
for $ 0 \leq s \leq t \leq t_0 $, where 
\begin{eqnarray*}
s_1 & = &
\int_0^s \tLambda (e^{(t-\sigma) \tLambda} - e^{(s-\sigma) \tLambda})(f(\sigma)-f(s)) d\sigma ,\\  s_2 & = &  (e^{t \tLambda}-e^{(t-s) \tLambda})(f(s)-f(t)) ,\\
s_3 & = &  \int_s^t \tLambda e^{(t-\sigma) \tLambda} (f(\sigma)-f(t)) d\sigma.
\end{eqnarray*}
We need an  estimate of the form 
$$ 
\| \tLambda R_1(t) - \tLambda R_1(s) \|_{\cX}  \leq \| s_1 \|_{\cX} + \| s_2 \|_{\cX} + \| s_3 \|_{\cX} ,
$$
where the $ \| s_j \|_{\cX}  $ for $ j = 1,2,3 $ have to be estimated by  $ C(t-s)^{\alpha} $.
\medskip

{\bf i)} We start with $ s_3 $. Using
$ \| \tLambda e^{(t - \sigma)\tLambda } \|_{{\cX}  \to {\cX} } \leq M_1
\remo{e^{-\beta_-(t - \sigma)}}
 (t - \sigma)^{-1}$, 
for a constant $ M_1 > 0 $,
and 
\begin{equation} \label{Lip}
 \| f(\sigma)-f(t) \|_{\cX}  \leq  \| f \|_{C^{\alpha}(\cX)}  (t-\sigma)^{\alpha} 
\end{equation} 
yields 
$$ 
\| s_3 \|_{\cX}   \leq M_1 \int_s^t e^{-\beta_-(t - \sigma)}(t-\sigma)^{\alpha-1} d\sigma \| f \|_{C^{\alpha}(\cX)}
\leq \frac{M_1}{\alpha} (t-s)^{\alpha}
\| f \|_{C^{\alpha}(\cX)} .
$$

{\bf ii)} Next we consider $ s_2 $. Using $ \| e^{t \tLambda} \|_{{\cX}  \to {\cX} } +  \| e^{(t-s) \tLambda} \|_{{\cX}  \to {\cX} }\leq 2 M_0 $, 
for a constant $ M_0 > 0 $,
and \eqref{Lip}
immediately yields the estimate
$$ 
\| s_2 \|_{\cX}  \leq 2 M_0 (t-s)^{\alpha} \| f \|_{C^{\alpha}(\cX)} .
$$

{\bf iii)} Finally we consider $ s_1 $.
Using 
$ \| \tLambda^2 e^{\tau \tLambda } \|_{{\cX}  \to {\cX} } \leq M_2 \tau^{-2}$,
for a constant $ M_2 > 0 $,
and \eqref{Lip}
yields the estimate
\begin{eqnarray*}
\| s_1 \|_{\cX}  & \leq &  
\int_0^s \| \tLambda (e^{(t-\sigma) \tLambda} - e^{(s-\sigma) \tLambda})\|_{L(\cX)}
\|f(\sigma)-f(s)\|_{\cX}  d\sigma
\\ & \leq &
 \int_0^s 
\| \tLambda^2 \int_{s-\sigma}^{t-\sigma} e^{\tau \tLambda} d\tau \|_{L(\cX)} \| f \|_{C^{\alpha}(\cX)}  (s-\sigma)^{\alpha} d\sigma
\\ & \leq & M_2 \int_0^s (s-\sigma)^{\alpha}
\int_{s-\sigma}^{t - \sigma} \tau^{-2} d\tau d\sigma\| f \|_{C^{\alpha}(\cX)} 
\\ & \leq &
  C M_2 \int_0^s  \int_{s-\sigma}^{t - \sigma} \tau^{\alpha-2} d\tau d \sigma  \| f \|_{C^{\alpha}(\cX)} 
\\ & \leq &
 \frac{C M_2}{\alpha(1-\alpha)} (t-s)^{\alpha}
\| f \|_{C^{\alpha}(\cX)} .
\end{eqnarray*}
Therefore, with i), ii), and iii) we have that
$ \tLambda R_1 $ is $ \alpha$-H\"older continuous in $[0,T_0/\varepsilon^2]$.
\medskip

{\bf b)} Next we come to the $ \alpha$-H\"older continuity of  the $ R_2 $-component  which we rewrite into
\begin{eqnarray*}
\tLambda R_2(t) & = &  e^{t \tLambda} (\tLambda R_0 + f(0)) +  e^{t \tLambda} (f(t) - f(0)) - f(t).
\end{eqnarray*}
Therefore, we have 
\begin{eqnarray*}
\lefteqn{\tLambda R_2(t) - \tLambda R_2(s)} \\
 & = &  (e^{t \tLambda} - e^{s \tLambda})(\tLambda R_0 + f(0)) 
\\ &&  +  e^{t \tLambda} (f(t) - f(0)) - f(t) - e^{s \tLambda} (f(s) - f(0)) + f(s)
\\ & = & s_4 + s_5 + s_6 ,
\end{eqnarray*}
with 
\begin{eqnarray*}
s_4  & = & (e^{t \tLambda} - e^{s \tLambda})(\tLambda R_0 + f(0)) , \\ 
s_5 & = & (e^{t \tLambda} - e^{s \tLambda} ) (f(s)-f(0)), \\
s_6 & = & (e^{t \tLambda} - 1)(f(t)-f(s)).
\end{eqnarray*}
As above, we need an estimate of the form 
$$ 
\| \tLambda R_2(t) - \tLambda R_2(s) \|_{\cX}  \leq \| s_4 \|_{\cX} + \| s_5 \|_{\cX} + \| s_6 \|_{\cX} ,
$$
where the $ \| s_j \|_{\cX}  $ for $ j = 4,5,6 $ have to be estimated by  $ C(t-s)^{\alpha} $.
\medskip
\begin{enumerate}[label = \roman*)]
  \item  We estimate
\begin{eqnarray*}
\| s_4 \|_{\cX} 
& \leq & \int_s^t \| \tLambda e^{\sigma \tLambda} \|_{L(D_{\tLambda}(\alpha,\infty),{\cX} )} d\sigma 
\| \tLambda R_0 + f(0) \|_{D_{\tLambda}(\alpha,\infty)} \\
& \leq & M_{1,\alpha}\int_s^t \sigma^{\alpha-1} d\sigma  \| \tLambda R_0 + f(0) \|_{D_{\tLambda}(\alpha,\infty)} \\
& \leq &  \frac{M_{1,\alpha}}{\alpha} \| \tLambda R_0 + f(0) \|_{D_{\tLambda}(\alpha,\infty)}
(t-s)^{\alpha}
\end{eqnarray*}

\item Next we obtain
\begin{eqnarray*}
\| s_5 \|_{\cX} 
& \leq & s^{\alpha} \| \tLambda \int_s^t   e^{\sigma \tLambda} d\sigma \|_{L(\cX)}
\| f \|_{C^{\alpha}(\cX)}\\ 
& \leq & M_1 \int_s^t  s^{\alpha} \sigma^{-1} d\sigma \| f \|_{C^{\alpha}(\cX)}\\ 
& \leq & M_1 s^{\alpha} (\ln t - \ln s)  \| f \|_{C^{\alpha}(\cX)} \\
& = & M_1  s^{\alpha} \ln \left(\frac{t}{s}\right) \| f \|_{C^{\alpha}(\cX)}\\ 
& = & M_1 s^{\alpha} \ln \left(1 + \frac{t-s}{s}\right) \| f \|_{C^{\alpha}(\cX)}.
\end{eqnarray*}
Now, if $t-s \leq s$, we can use the estimate $\ln(1 + y) \leq y$ for $y \geq 0$ to obtain $s^\alpha \ln\left(1 + \frac{t-s}{s}\right) \leq s^{\alpha -1} (t-s) \leq (t-s)^\alpha$. Otherwise if $t-s \geq s$, we use the estimate $\ln(1 + y) \leq C y^{\alpha/2}$ for $y \geq 0$ to obtain $s^\alpha \ln\left(1 + \frac{t-s}{s}\right) \leq s^{\alpha/2} (t-s)^{\alpha/2} \leq (t-s)^\alpha$. In any case we obtain the wanted
\begin{equation}
  \| s_5 \|_{\cX}  \leq CM_1 (t-s)^\alpha \| f \|_{C^{\alpha}(\cX)}.
\end{equation}

\item Moreover, we find
\begin{eqnarray*}
\| s_6 \|_{\cX} 
& \leq &
 (M_0 + 1)(t-s)^{\alpha}\| f \|_{C^{\alpha}(\cX)}
\end{eqnarray*}

\end{enumerate}

Therefore, with i), ii), and iii) we have that
$ \tLambda R_2 $ is $ \alpha$-H\"older continuous in $[0,t_0]$.
\medskip

{\bf c)} The estimates a) and b) finally 
imply
\begin{eqnarray*}
\| R \|_{\cZ^1} 
 \leq C  (\| f \|_{\cZ^0} + \| R_0 \|_{D(\Lambda)}+
\|\tLambda R_0 + f |_{t=0}\|_{D_{\tLambda} (\alpha,\infty)})
\end{eqnarray*}
for all $ \widetilde{\eta}  \geq h $. \qed
\medskip

{\bf Proof of Lemma \ref{lemestimateEcpart}.}
As already said, the  estimates for the $ \tR_{c,\pm 1} $-parts follow from classical 
semigroup theory for semilinear systems due to the 
compact support in Fourier space of the multipliers $ \widehat{E}_{c,\pm 1} $.
This compact support immediately implies 
$$
E_{c,\pm 1} f \in  C^{\alpha}([0,T_0/\varepsilon^2], H^q_{l,u})
$$ 
for any  $q \geq 0 $.
From the variation of constant formula
$$ 
\tR_{c,\pm 1}(t)  =  e^{t \tLambda} \tR_{c,\pm 1}(0) + \int_0^t e^{(t-s) \tLambda} E_{c,\pm 1} \tf(s) ds,
$$ 
recalling the semigroup estimate \eqref{sg_estimate} for $e^{t \Lambda}$, we obtain 
$$
\tR_{c,\pm 1} \in  C^{0}([0,T_0/\varepsilon^2],H^{q}_{l,u}),
$$ 
with the estimate 
\begin{eqnarray*}
\lefteqn{\sup_{t \in [0,T_0/\varepsilon^2]} \| \tR_{c,\pm 1}(t) \|_{H^{q}_{l,u}} }\\
& \leq & C \| \tR_{c,\pm 1}(0) \|_{H^{q}_{l,u}} + C \sup_{t \in [0,T_0/\varepsilon^2]} \left|\int_0^t e^{- (\widetilde{\eta} - h) \varepsilon^{2} (t-s)} \|E_{c,\pm 1} \tf(s) \|_{H^{q}_{l,u}} ds\right|
\\
& \leq & C \| \tR_{c,\pm 1}(0) \|_{H^{q}_{l,u}} + C (\widetilde{\eta} - h)^{-1} \varepsilon^{-2} \sup_{t \in [0,T_0/\varepsilon^2]} \| E_{c,\pm 1} \tf(t) \|_{H^{q}_{l,u}}.
\end{eqnarray*}
This implies $
\tLambda \tR_{c,\pm 1} \in  C^{0}([0,T_0/\varepsilon^2],H^{q-4}_{l,u})
$ with 
 $$
\sup_{t \in [0,T_0/\varepsilon^2]} \| \tLambda \tR_{c,\pm 1}(t) \|_{H^{q-4}_{l,u}} 
\leq C \| \tR_{c,\pm 1}(0) \|_{H^{q}_{l,u}} + C (\widetilde{\eta} - h)^{-1} \varepsilon^{-2} \sup_{t \in [0,T_0/\varepsilon^2]} \| E_{c,\pm 1} \tf(t) \|_{H^{q}_{l,u}}.
$$
From $ \partial_t \tR_{c,\pm 1}  =  \tLambda \tR_{c,\pm 1} + E_{c,\pm 1} \tf $  
we immediately find 
\begin{eqnarray*}
\lefteqn{\sup_{t \in [0,T_0/\varepsilon^2]} \| \partial_t  \tR_{c,\pm 1}(t) \|_{H^{q-4}_{l,u}} }\\
&\leq &C \| \tR_{c,\pm 1}(0) \|_{H^{q}_{l,u}} + C ((\widetilde{\eta} - h)^{-1} \varepsilon^{-2} +1)\sup_{t \in [0,T_0/\varepsilon^2]} \| E_{c,\pm 1} \tf(t) \|_{H^{q}_{l,u}},
\end{eqnarray*}
and so 
$ \tR_{c,\pm 1} \in C^{1}([0,T_0/\varepsilon^2],H^{q-4}_{l,u}) \subset C^{\alpha}([0,T_0/\varepsilon^2],H^{q-4}_{l,u}) $. Since then 
$
\tLambda \tR_{c,\pm 1} \in  C^{\alpha}([0,T_0/\varepsilon^2],H^{q-8}_{l,u})
$ 
with 
\begin{eqnarray*}
\lefteqn{\| \tLambda \tR_{c,\pm 1} \|_{C^{\alpha}([0,T_0/\varepsilon^2],H^{q-8}_{l,u})}}
\\ & \leq & C \| \tR_{c,\pm 1}(0) \|_{H^{q}_{l,u}} 
+ C ((\widetilde{\eta} - h)^{-1} \varepsilon^{-2} +1 ) \| E_{c,\pm 1} \tf \|_{C^{\alpha}([0,T_0/\varepsilon^2],H^{q}_{l,u})}
\end{eqnarray*}
the differential equation implies
$ 
\partial_t \tR_{c,\pm 1} \in  C^{\alpha}([0,T_0/\varepsilon^2],H^{q-8}_{l,u})
$ 
with the same estimates and so 
$$ 
\tR_{c,\pm 1} \in  C^{1+\alpha}([0,t_0],H^{q-8}_{l,u})
$$ 
with the same estimates. 
Since  this holds for all $ q \geq 0 $, we can choose $ q \geq r + 8 $. 
Therefore, we are done. \qed

\section{The quasilinear B\'enard problem}

\label{sec3}
The second example which we will consider is a quasilinear version of B\'enard's problem.
In contrast to the Swift-Hohenberg equation B\'enard's problem contains quadratic terms which forces us to modify the proof of the approximation result like in the semilinear case.
The Ginzburg-Landau approximation for the classical (semilinear) B\'enard problem is  
justified in \cite{Schn94ZAMP}.
In the following we consider a simple quasilinear adaption, 
similar to \cite{ZX22},
of this well-studied problem
such that we can concentrate on the handling of the
quasilinear feature and refer to \cite{Schn94ZAMP}
for the derivation of the Ginzburg-Landau equation and for the estimates
of the residual terms. 

The physical set-up of 
B\'enard's problem consists of a fluid contained between two plates, where the
lower plate is heated and the upper plate is cooled.
In detail,
we follow \cite{Schn94ZAMP} and consider B\'enard's problem in a strip,
$ (x,y) \in   \R \times (0,\pi) $,
with temperature field $ T $, pressure field $ p $, and velocity field $ u = (u_1,u_2) $. 
The so-called Boussinesq approximation is used, i.e.,
the density $ \rho $ is considered to
be a constant except in the
buoyancy term which depends in an affine manner  on the temperature, i.e.
$
\rho(T) = c_0 + c_1 T 
$, with $ c_0 $ and $ c_1 $ some constants.
The problem is supplemented with  
 the mean
flux condition $ \int^\pi_0 u_1 d y = 0$ and
the boundary conditions
$$
 \partial_y u_1|_{y=0} = u_2|_{y=0} =  \partial_y u_1|_{y=\pi} = u_2|_{y=\pi} = 0, 
 \quad T|_{y=0}= T_0 , \quad T|_{y=\pi} = T_1 ,
$$
where $ T_0 \geq  T_1 $. 
These boundary conditions allow an explicit stability analysis. 
In contrast to the classical case where 
the heat diffusion coefficient and the viscosity  are constants here,
we allow them to depend on the velocity field $ u $.

The equation for the diffusion and transport of heat is then given 
$$ 
\partial_t T  = \delta_T(u) \Delta T   -  (u\cdot \nabla) T,
$$ 
with $ u $-dependent heat diffusion coefficient $ \delta_T(u)  > 0 $.
The pure heat conduction state, i.e. $ u = 0 $,  depends only on $ y $ 
and satisfies $ \partial_y^2 T = 0 $.
Therefore, like in the classical case the  pure heat conduction 
state has an affine temperature profile, namely 
$$T = T_0 {+} y(T_1{-} T_0)/\pi. $$

We introduce the deviation $ \Theta = T- T_0 {-}
y(T_1{-} T_0)/\pi $ from the linear heat profile, 
satisfying the boundary conditions
$$ 
 \Theta|_{y=0} = \Theta|_{y=\pi} = 0 .
$$ 
Thus, we finally consider the  slightly modified Oberbeck-Boussinesq system
\begin{eqnarray} \label{b1}
 \partial_t u & = & (1+ \mu_{1}(u)) \Delta u - \nabla p -  \rho \Theta\vec{e}_2  -(u\cdot \nabla)u,\\
 \partial_t \Theta & = & (\kappa + \mu_{2}(u))  \Delta \Theta + u_2  -(u\cdot \nabla)\Theta,\\
0 & = & \nabla\cdot u. \label{b3}
\end{eqnarray}
The equations contain two dimensionless parameters, namely the Rayleigh number 
$ \rho = \beta_0 (T_0-T_1)h^3/(\pi^3 \nu^2) $ and $  \kappa = \delta /\nu $,
where $ \delta $ stands for the heat conductivity for $ u = 0 $, $ \nu $ for the viscosity for $ u =0 $, $ \beta_0 $ 
for the buoyancy  parameter, and $ h $ for the physical height of the fluid. 
Hence, we assume that $ \mu_1(0) = \mu_2(0) = 0$.

The pure heat conduction state $ (u,\Theta,p) = (0,0,0) $ is stable if the temperature difference 
$ T_1-T_0 $ between the 
lower and the upper plate is sufficiently small.  
The trivial solution $ (u,\Theta,p) = (0,0,0) $
loses stability and convection sets in
if the  
temperature difference is sufficiently large. 
Close to threshold of instability we derive a Ginzburg-Landau equation 
to describe the bifurcating solutions.


\subsection{Linear stability analysis}

We recall and adapt the derivation of the Ginzburg-Landau approximation
from \cite{Schn94ZAMP}.
The linearization around the pure heat conduction state $ (u,\Theta,p) = (0,0,0) $ is given by
\begin{eqnarray*}
\partial_t u & = &  \Delta u - \nabla p -  \rho \Theta\vec{e}_2,  \\
\partial_t \Theta & = & \kappa \Delta \Theta + u_2, \\
0 & = & \nabla\cdot u,
\end{eqnarray*}
which is solved by 
$$ 
\left(
\begin{array}{c} u_1 \\u_2 \\ \Theta \\ p \end{array}
\right)
(x,y,t) = e^{\lambda_{m,\pm}(k) t} e^{i kx}
\left( \begin{array}{c}
\widehat{u}_{1,m}\cos(my) \\ \widehat{u}_{2,m}\sin(my) \\ \widehat{\Theta}_{m}\sin(my) \\ \widehat{p}_{m}\cos(my) \end{array}\right)
$$
for $ m \in \{ 1,2,3,\ldots \} $. 
The vanishing mean flux condition excludes $ m = 0 $.
We find the curves of eigenvalues
$$ 
2 \lambda_{m,\pm}(k) = -(\kappa +1)s \pm \sqrt{(\kappa - 1)^2s^2+ 4 \rho k^2 s^{-1}},
$$
with $ s = k^2 + m^2 $.
The associated eigenfunctions in the $ (u_1,u_2,\Theta) $-variables 
satisfy $ \nabla\cdot u = 0 $ and 
are denoted by $ 
\varphi_{m,\pm}(k) $.

By fixing $ \kappa $ and varying the control parameter $ \rho $ we find  that the trivial solution is spectrally stable for 
$ \rho < \rho_{c} =27\kappa/4 $. 
For $ \rho_{c} =27\kappa/4 $ we have $  \lambda_{1,+}(\pm k_c) = 0 $
for $ \pm k_c = \pm 1/\sqrt{2} $ and so a Turing instability occurs.
For the description of the bifurcating solutions we derive a Ginzburg-Landau
equation which is the universal amplitude equation which occurs in case of Turing instabilities.
To do so, we introduce the small bifurcation parameter $ 0 < \varepsilon^2 \ll 1 $
by $ \rho = \rho_{c}+ \varepsilon^2 $. Note that $ \rho $ can physically be 
controlled through the temperature difference $ T_1-T_0 $ between the upper and lower plate.
For $  \varepsilon^2 > 0 $ the curve of eigenvalues  
$  \lambda_{1,+}(k) $ is positive 
in an $ \varepsilon $-neighborhood around
$ \pm k_c = \pm 1/\sqrt{2} $. See \cite{Schn94ZAMP} for more details.

\subsection{The time-dependent quasilinear Oberbeck-Bous\-sinesq system}

We recall now how to handle the Navier-Stokes equations as a dynamical system.
To get rid of the equation $ 0  =  \nabla\cdot u $ and of the variable $ p $ which appears 
without time derivative, we interpret $ \nabla p $ as projection $ P $ on the space of functions $ u $
satisfying $ 0  =  \nabla\cdot u $. 
In detail, we introduce the space 
\begin{eqnarray} \label{Xdefben}
\cX & = & \{ (u,\Theta) \in  (L^2_{l,u}(\R \times (0,\pi)))^3 :  
\\ && \qquad 0  =  \nabla\cdot u ,  u_2|_{y=0,\pi}  =  u \cdot n |_{y=0,\pi} = 0,
 \int_0^{\pi} u_1 dy = 0  \}, \nonumber
\end{eqnarray}
where $ n = (0, \pm 1)^T$ is the outer normal at the strip.
We have 
\begin{lemma} \label{proj}
The mapping $ P: f \mapsto u $, defined through the unique solution of 
$$ 
0  =  \nabla\cdot u , \qquad u + \nabla p = f , \qquad u_2|_{y=0,\pi}  = 0 , \qquad \int_0^{\pi} u_1 dy = 0,
$$ 
is a smooth mapping from $ (L^2_{l,u}(\R \times (0,\pi)))^3 $ to $ \cX $ with 
$ \| P f \|_{L^2_{l,u}} \leq \| f \|_{L^2_{l,u}} $.
\end{lemma} 
\noindent
{\bf Proof.} See \cite[\S 4 (A1)]{Schn94ZAMP}.\qed
\medskip

With the help of this projection we rewrite the quasilinear Oberbeck-Boussinesq system as
evolution problem
\begin{equation} \label{OBS}
\partial_t U = \Lambda U + B(U,U)
\end{equation}
where $ U = (u,\Theta)^T $, and
$$ 
\Lambda U = P \left( \begin{array}{c} \Delta u -  \rho \Theta\vec{e}_2 \\ 
\kappa  \Delta \Theta + u_2 \end{array} \right), \qquad 
 B(U,U) = P \left( \begin{array}{c} \mu_{1}(u) \Delta u  -(u\cdot \nabla)u \\
 \mu_{2}(u)  \Delta \Theta -(u\cdot \nabla)\Theta 
  \end{array} \right),
$$
In the following, for notational simplicity,  we assume that $ \mu_1 $ and $ \mu_2 $ depend 
linearly on $ u $ such that $ B $ can be interpreted as symmetric bilinear mapping.
From \cite[\S 4 (A2)]{Schn94ZAMP} it is known that
the domain of definition $ D(\Lambda) $ of $ \Lambda $ is given by 
\begin{eqnarray} \label{Ddefben}
D(\Lambda) & = & \{ (u,\Theta) \in  (H^2_{l,u}(\R \times (0,\pi)))^3 :  
\\ && \qquad 0  =  \nabla\cdot u ,  u_2|_{y=0,\pi}  =  \partial_y u_1|_{y=0,\pi} = \Theta|_{y=0,\pi} = 0,
 \int_0^{\pi} u_1 dy = 0  \} \nonumber
\end{eqnarray}
and that $ \Lambda : D(\Lambda) \to \cX $ is a sectorial operator generating an 
analytic semigroup $ (e^{t \Lambda})_{t \geq 0} $ in $ \cX $ and there exist $h \geq 0$, $C > 0$ independent of $\varepsilon$ such that for all $t > 0$
\begin{equation}
  \|e^{t\Lambda}\|_{\cX \to \cX} \leq Ce^{h\varepsilon^2 t}.
\end{equation}
Since $ H^2_{l,u}(\R \times (0,\pi)) \subset L^{\infty}(\R \times (0,\pi)) $ it is easy to see by Lemma 
\ref{proj} that 
$ B : D(\Lambda)  \times D(\Lambda)  \to \cX $ is a smooth bilinear mapping. In detail,
there exists 
a $ C_B > 0 $ such that   
\begin{equation} \label{Besti}
\| B(U,V) \|_{\cX} \leq C_B \| U \|_{D(\Lambda)}\| V \|_{D(\Lambda)}
\end{equation}
for all $ U,V \in D(\Lambda) $.
Following the approach of Section \ref{sec2} we introduce two  spaces
to get rid of the quasilinear character of the equations, namely 
$$
X^{\alpha}_{\eta,t_0} =  C^{\alpha}([0,t_0],\cX)
$$
and 
$$
Y^{\alpha}_{\eta,t_0} = C^{\alpha}([0,t_0],D(\Lambda)) \cap C^{1+\alpha}([0,t_0],\cX).
$$
They are equipped with the norms
$$
\| R \|_{X^{\alpha}_{\eta,t_0}} = \sup_{t \in [0,t_0]} \| e^{- \eta t} R(t) \|_{C^{\alpha}([0,t_0],\cX)}
$$
and 
$$
\| R \|_{Y^{\alpha}_{\eta,t_0}} = \sup_{t \in [0,t_0]} \| e^{- \eta t} R(t) \|_{C^{\alpha}([0,t_0],D(\Lambda))} + \sup_{t \in [0,t_0]} \| e^{- \eta t} R(t) \|_{C^{1,\alpha}([0,t_0],\cX)}.
$$
As above we choose $ \eta = \widetilde{\eta} \varepsilon^2 $
 with $ \widetilde{\eta}  > 1 $ fixed, sufficiently large, and independent of $ 0 < \varepsilon^2 \ll 1 $.
To use the previous proofs also in this more general situations we 
set 
\begin{equation} \label{Zdefben}
\cZ^0 = X^{\alpha}_{\widetilde{\eta} \varepsilon^2,T_0/\varepsilon^2}  \qquad \textrm{and} \qquad \cZ^1 = Y^{\alpha}_{\widetilde{\eta} \varepsilon^2,T_0/\varepsilon^2}  .
\end{equation}
Since $ e^{\widetilde{\eta} \varepsilon^2 t} = \mathcal{O}(1) $ for 
 $ t \in [0,T_0/\varepsilon^2]$ the estimate \eqref{Besti} transfers into:
There exists 
a $ C_B > 0 $ such that   for all $ \varepsilon \in (0,1) $ we have
\begin{equation} \label{Besti2}
\| B(U,V) \|_{\cZ^0} \leq C_B e^{\widetilde{\eta} T_0}\| U \|_{\cZ^1}\| V \|_{\cZ^1}
\end{equation}
for all $ U,V \in \cZ^1 $.

\subsection{The Ginzburg-Landau approximation}

For the derivation of a Ginzburg-Landau equation we make the ansatz
$$ 
\left(\begin{array}{c} u \\ \Theta \end{array} \right)  = \varepsilon \Psi_{GL}(x,y,t)+\ \mathcal{O}(\varepsilon^2)
= \varepsilon A(X, T) e^{i k_c x}\varphi_{1,+}(k_c) + c.c. + \mathcal{O}(\varepsilon^2), 
$$ 
with $ X = \varepsilon x $ and $ T = \varepsilon^2 t $.
With $ \varphi_{m,\pm}(k) $ we denote the normalized eigenfunction
associated to the eigenvalue $ \lambda_{m,\pm}(k) $.
Inserting this into \eqref{b1}-\eqref{b3} we find  that the complex amplitude $ A $, which modulates the spatially periodic pattern 
$ e^{i  k_c x}\varphi_{1,+}(k_c) $ 
slowly in time and in space, has to satisfy the GL equation
\begin{equation} \label{GLbenard}
\partial_T A = \frac{4 \kappa}{\kappa+1} \partial_X^2 A+\frac{2}{9(\kappa+1)} A - \gamma A|A|^2,
\end{equation}
with a $ \gamma \in \R $ for which in general 
no simple formula exists. In the classical semilinear case, 
\eqref{b1}-\eqref{b3} can be extended periodically into the $ y $-direction.
Then for the calculation of 
the cubic coefficient only finitely many modes, namely $ m \in \{1,2\} $ 
and $ k \in \{-2k_c, -k_c, 0, k_c, 2k_c\} $, play a role and an explicit 
calculation of $ \gamma $ is possible. In our quasilinear situation 
this is possible if the new functions $ \mu_1(u) $ and $ \mu_2(u) $ would be 
even with respect to $ y $. 
Then $ u_1 $ and $ p $ could be expanded in a series in $ \cos(my) $ 
and $ u_2 $ and $ \Theta $  in a series in $ \sin(my) $,
respectively.
%
%

We refer again to \cite{Schn94ZAMP} for a detailed derivation 
of \eqref{GLbenard}.  

\subsection{The mode filters}

We follow the existing Ginzburg-Landau approximation theory
for semilinear systems
and separate the critical modes from the stable ones.
For defining mode filters, doing this task, in the vector-valued case 
we introduce first a smooth cut-off function $ \chi_{\pm} \in C_0^{\infty}$
defined by
$$ 
\widehat{\chi}_{\pm}(k) = \left\{ \begin{array}{cl} 1, & k   \in  [\pm k_c- k_c/20,\pm k_c + k_c/20] , \\
0 ,&  k \not   \in  (\pm k_c- k_c/10,\pm k_c + k_c/10) , \\
\in [0,1] , & \textrm{elsewhere}.
\end{array} \right.
$$
The mode filters on the critical modes are then defined through
$$  
E_{c,\pm 1} = \cF^{-1} \widehat{E}_{c,\pm 1} \cF
$$ 
where 
$$ 
\widehat{E}_{c,\pm 1}(k) = \widehat{\chi}_{\pm}(k) \langle \varphi^*_{1,+}(k),  \cdot \rangle_{L^2(0,\pi)} \varphi_{1,+}(k)
$$ 
with $ \varphi^*_{1,+}(k) $ being the eigenfunction to the adjoint 
operator $ \widehat{\Lambda}^*(k)$ associated to the eigenvalue $\lambda_{1, +}(k)$ and normalized 
with $ \langle \varphi^*_{1,\pm}(k),  \varphi_{1,\pm}(k)\rangle_{L^2(0,\pi)} = 1 $.
We define the mode filters on the stable complement by
$$ 
E_s = P( I - E_{c,+}-E_{c,-}).
$$  

\subsection{Estimates for the residual}

As in \cite{Schn94ZAMP}  the  Ginzburg-Landau  approximation
$ \varepsilon \Psi $, used for the derivation of the 
Ginzburg-Landau  equation, gives a residual 
$$ 
\Res(U) = - \partial_t U + \Lambda U + B(U,U),
$$
which is sufficiently small for our purposes.
\begin{lemma} \label{lem51a}
Let $ A \in C([0,T_0],H^3_{l,u}) $ be a solution of the 
Ginzburg-Landau equation \eqref{GL} satisfying 
\begin{equation} \label{eq19}
\sup_{T \in [0,T_0/\varepsilon^2]} \| A(\cdot,T) \|_{H^3_{l,u}}
\leq C_{GL}.
\end{equation}
Then   there exist $ \varepsilon_0 > 0 $ 
and $ C_2 > 0 $, only depending on $ C_{GL} > 0 $, $ T_0 > 0 $ 
such that for all $ \varepsilon \in (0,\varepsilon_0)
$ we have 
$$
\sup_{t \in [0,T_0/\varepsilon^2]} \| E_c \Res(\varepsilon \Psi) \|_{\mathcal{X}} \leq C_2 \varepsilon^{7/2} 
$$ 
and 
$$
\sup_{t \in [0,T_0/\varepsilon^2]} \| E_s \Res(\varepsilon \Psi) \|_{\mathcal{X}} \leq C_2 \varepsilon^{5/2}.
$$ 
\end{lemma}
\begin{remark}{\rm
By assuming $A \in C([0, T_0], H^{3 + 2\alpha}_{l, u})$,
as in Remark \ref{remmakre}, we obtain the associated 
estimates for the residual in the $ \cZ^0 $-space.
}
\end{remark}

\subsection{The equations for the error}

We write the solution $ U $ 
of \eqref{OBS} as a sum of the  approximation $ \varepsilon \Psi $ 
and an error $  \varepsilon^2 R $. We split the approximation as
$$ 
\varepsilon \Psi = \varepsilon \Psi_c + \varepsilon^2 \Psi_s,
$$ 
where $\Psi_c = E_c \Psi$ and $\varepsilon \Psi_s = E_s \Psi$.  Thus, the support of $ \Psi_c $ in Fourier space is contained in the support of $ \widehat{E}_c $ 
and that of $ \Psi_s $ is contained in the support of $ \widehat{E}_s $.
Moreover, we split the error $\varepsilon^{3/2} R = V - \varepsilon \Psi$, as
$$ 
\varepsilon^{3/2} R = \varepsilon^{3/2} R_c + \varepsilon^{5/2} R_s,
$$ 
where $ R_c = E_c R $ and $ \varepsilon R_s = E_s R $. Thus, the support of $R_c$ in Fourier space is contained in the support of $ \widehat{E}_c $ 
and that of $ R_s $ is contained in the support of $ \widehat{E}_s $.
We define  $ R_c $ and $ R_s $ to be  solutions of 
\begin{eqnarray} \label{erreq1}
\partial_t R_c & = & \Lambda R_c + L_{2,c}(\varepsilon \Psi,\varepsilon^2 R)
+ N_{2,c}(\varepsilon \Psi,\varepsilon^2 R)  
 + \varepsilon^{-3/2} E_c \Res(\varepsilon \Psi), 
 \\ \label{erreq2}
\partial_t R_s & = & \Lambda R_s + L_{2,s}(\varepsilon \Psi,\varepsilon^2 R)
+ N_{2,s}(\varepsilon \Psi,\varepsilon^2 R)  + \varepsilon^{-5/2} E_s \Res(\varepsilon \Psi), 
\end{eqnarray}
where 
\begin{eqnarray*}
L_{2,c}(\varepsilon \Psi,\varepsilon^2 R) & = & 2 \varepsilon^2  E_c B(\Psi_c, R_s) 
+ 2 \varepsilon^2  E_c B(\Psi_s, R_c)  + 2 \varepsilon^3  E_c B(\Psi_s, R_s) , \\ 
N_{2,c}(\varepsilon \Psi,\varepsilon^2 R) & = & 2 \varepsilon^{5/2}  E_c B(R_c, R_s) + \varepsilon^{7/2}  E_c B(R_s, R_s), \\
L_{2,s}(\varepsilon \Psi,\varepsilon^2 R) & = & 2   E_s B(\Psi_c, R_c) + 2 \varepsilon  E_s B(\Psi_c, R_s)   \\ && \qquad  + 2 \varepsilon  E_s B(\Psi_s, R_c)  + 2 \varepsilon^2  E_s B(\Psi_s, R_s) ,\\
N_{2,s}(\varepsilon \Psi,\varepsilon^2 R) & = &  \varepsilon^{1/2}  E_s B(R_c, R_c) + 2 \varepsilon^{3/2}  E_s B(R_c, R_s) + \varepsilon^{5/2}  E_s B(R_s, R_s).
\end{eqnarray*}
Note that $ E_c B(\Psi_c,R_c) =  E_c B(R_c,R_c) = 0 $ due to disjoint supports 
of $ E_c $ on the one hand and $ B(\Psi_c,R_c)  $ and $ B(R_c,R_c)  $ on the other hand 
in Fourier space.
Using  the fact that $ \varepsilon \Psi $ is uniformly bounded in $ D(\Lambda) $ for
$ t \in [0,T_0/\varepsilon^2] $ and \eqref{Besti2}
we easily find 
\begin{eqnarray} \label{eq36}
 \| L_{2,c}(\varepsilon \Psi,\varepsilon^2 R)
+ N_{2,c}(\varepsilon \Psi,\varepsilon^2 R)   \|_{\cZ^0} 
 & \leq  & C \varepsilon^2 \| R \|_{\cZ^1}  + C  \varepsilon^{5/2} e^{\widetilde{\eta} T_0} \| R \|_{\cZ^1}^2 , \\ \label{eq37}
 \| L_{2,s}(\varepsilon \Psi,\varepsilon^2 R)
+ N_{2,s}(\varepsilon \Psi,\varepsilon^2 R)   \|_{\cZ^0} 
 & \leq &  C  \| R_c \|_{\cZ^1}  + C \varepsilon \| R_s \|_{\cZ^1}  \\ && + C  \varepsilon^{1/2} 
 e^{\widetilde{\eta} T_0} \| R \|_{\cZ^1}^2 . \nonumber
\end{eqnarray}

The counterparts to Lemma \ref{lemestimateEcpart} and Lemma \ref{lemestimateErpart} are as follows:
\begin{lemma}
\label{lemestimateEcpartben}
For all  $ \alpha \in (0,1) $,  $ T_0 > 0 $,  there exists a   
$ C > 0 $ such that for all $ \varepsilon \in (0,1) $ the following holds.
Let
$ f \in \cZ^0 $,
$ R_0  \in D(\Lambda) $, and
let $ R $ be the mild solution of  
\begin{equation} \label{eq401g}
\partial_t R_{c,\pm 1} = \Lambda R_{c,\pm 1} + E_{c,\pm 1} f  
\end{equation}
with $ R|_{t = 0} = {E}_{c,\pm 1} R_0 $.
Then we have that the solution $ R_{c,\pm 1} = \mathcal{K} E_{c,\pm 1} f \in \cZ^1 $ satisfies
\begin{eqnarray*}
\| R_{c,\pm 1} \|_{\cZ^1} 
 & \leq & C ( (\widetilde{\eta} - h)^{-1} \varepsilon^{-2} + 1)
 (\| f \|_{\cZ^0} + \| R_0 \|_{D(\Lambda)})
\end{eqnarray*}
for all $ \widetilde{\eta}  >  h $.
\end{lemma}
\noindent
{\bf Proof.}
Since $ \widehat{E}_{c,\pm 1}(k) $ has a one-dimensional range for fixed $ k \in \R $ 
and a compact support in Fourier space we have that $ E_{c,\pm 1} $ maps 
$ \cX $ in every $ H^q_{l,u} $.
Therefore, 
the proof goes line for line as the proof of Lemma \ref{lemestimateEcpart}.
\qed

\begin{lemma}
\label{lemestimateErpartben}
For all  $ \alpha \in (0,1) $,  $ T_0 > 0 $,  there exists a   
$ C > 0 $ such that for all $ \varepsilon \in (0,1) $ the following holds.
Let
$ f \in \cZ^0 $,
$ R_0  \in D(\Lambda) $, and
let $ R_s $ be the mild solution of  
\begin{equation} \label{eq401gg}
\partial_t R_s = \Lambda R_s + E_s f   
\end{equation}
with $ R_s|_{t = 0} = E_s R_0 $.
Then we have that the solution $ R_s = \mathcal{K} E_s f \in \cZ^1$ satisfies
\begin{eqnarray*}
\| R_s \|_{\cZ^1} 
 & \leq & C  
 (\| f \|_{\cZ^0} + \| R_0 \|_{D(\Lambda)}+
\|\Lambda R_0 + f |_{t=0}\|_{D_{\Lambda} (\alpha,\infty)})
\end{eqnarray*}
for all $ \widetilde{\eta}  > 0 $.
\end{lemma}
\noindent
{\bf Proof.}
The proof goes as the proof of Lemma \ref{lemestimateErpart}, but  
with $ \cX $ defined in \eqref{Xdefben}, $ D(\Lambda) $ defined in \eqref{Ddefben},
$ \cZ^0 $ and $ \cZ^1 $ defined in \eqref{Zdefben}, and $ \mu = 2 $ at the beginning 
of the proof of Lemma \ref{lemestimateErpart}.
\qed
\medskip

The error estimates will follow with a fixed point argument. From 
\eqref{erreq1}-\eqref{erreq2} we obtain
\begin{eqnarray} \label{erreq1inv}
R_c & = & F_c( R_c,R_s),
\\ \label{erreq2inv}
R_s & = &  F_s( R_c,R_s),
\end{eqnarray}
where
\begin{eqnarray*} 
 F_c( R_c,R_s) & = & \mathcal{K}_c( L_{2,c}(\varepsilon \Psi,\varepsilon^2 R)
+ N_{2,c}(\varepsilon \Psi,\varepsilon^2 R)  
+ \varepsilon^{-2} E_c \Res(\varepsilon \Psi)), 
 \nonumber \\
 F_s( R_c,R_s) & = & \mathcal{K}_s ( L_{2,s}(\varepsilon \Psi,\varepsilon^2 R)
+ N_{2,s}(\varepsilon \Psi,\varepsilon^2 R)  + \varepsilon^{-3} E_s \Res(\varepsilon \Psi)).
\end{eqnarray*}
We prove that the mapping $ (R_c,R_s) \mapsto (F_c( R_c,R_s),F_s( R_c,R_s)) $
is a contraction in a ball in the space $ \cZ^1 \times \cZ^1 $ 
equipped with norm
$$ 
  \| R_c \|_{\cZ^1} + 
\beta_0 \| R_s \|_{\cZ^1} 
$$ 
where $ \beta_0 > 0 $ is suitably chosen below to get rid of the Jordan block structure of $ 
(F_c,F_s) $.
With $  Z =  \| R_c \|_{\cZ^1} + 
\beta_0 \| R_s \|_{\cZ^1}  $ we estimate, using  Lemma \ref{lemestimateEcpartben} and Lemma \ref{lemestimateErpartben}, using \eqref{eq36} and \eqref{eq37}, that
\begin{eqnarray*} 
 \| F_c( R_c,R_s)  \|_{\cZ^1} 
 & \leq & C ((\widetilde{\eta}-h)^{-1} + \varepsilon^2) \Big( \| R_c \|_{\cZ^1} + \| R_s \|_{\cZ^1} \\ 
 && + e^{ \widetilde{\eta} T_0} \varepsilon^{1/2} \left(\| R_c \|_{\cZ^1} + \| R_s \|_{\cZ^1}\right)^2 
+ C_{res} \Big)\\
 & \leq & C ((\widetilde{\eta}-h)^{-1} + \varepsilon^2) \left(\beta_0^{-1} Z + e^{2 \widetilde{\eta}} \varepsilon^{1/2} ( \beta_0^{-1} Z)^2 + C_{res}\right)
\end{eqnarray*}
and 
\begin{eqnarray*}  
\beta_0 \| F_s( R_c,R_s)  \|_{\cZ^1} 
 & \leq & C \beta_0 (
 \| R_c \|_{\cZ^1} + \varepsilon \| R_s \|_{\cZ^1}) \\ && \qquad + C e^{\widetilde{\eta} T_0}  \varepsilon^{1/2} (\| R_c \|_{\cZ^1} + \| R_s \|_{\cZ^1})^2 + C C_{res}
\\ &  \leq & 
 C \beta_0 Z + C \varepsilon Z + C e^{\widetilde{\eta} T_0}  \varepsilon^{1/2} Z^2 + C \beta_0 C_{res}. 
\end{eqnarray*}
Therefore, the right hand side of \eqref{erreq1inv}-\eqref{erreq2inv}
maps a ball of $ \cZ^1 \times \cZ^1 $ with fixed, but sufficiently large radius $ \rho_0 $ in itself for 
$ \beta_0> 0 $ chosen sufficiently 
small, then $   \widetilde{\eta}  $ chosen sufficiently large,  and finally $ \varepsilon > 0 $ chosen sufficiently 
small.
With the same argument  the right hand side of \eqref{erreq1inv}-\eqref{erreq2inv} can be 
shown to be a  contraction in this  ball.
Hence, the exists a unique fixed point in this ball for this mapping. 
\medskip

Therefore, we have established the following approximation result.

\begin{theorem}\label{appbenard}
Let $ \alpha \in (0,1)  $ and let  $ A  \in  C([0,T_0],H^{3+2 \alpha}_{l,u}(\R, \C)) $ be a solution of the GL equation 
\eqref{GLbenard} satisfying 
\begin{equation} \label{eq19s2ben}
\sup_{T \in [0,T_0/\varepsilon^2]} \| A(\cdot,T) \|_{H^{3+2\alpha}_{l,u}}
\leq C_{GL}.
\end{equation}
Then there exist $ \varepsilon_0 > 0 $ 
and $ C_2 > 0 $, only depending on 
$ C_{GL} > 0 $, $ T_0 > 0 $,  such that  for all
$ \varepsilon \in (0,\varepsilon_0) $ there are solutions 
$ (u,\Theta) $ of the modified Oberbeck-Boussinesq system
\eqref{b1}-\eqref{b3}
 with
 $$ 
\sup_{t \in [0,T_0/\varepsilon ^2]} \|(u,\Theta)(\cdot,\cdot,t)- \varepsilon \Psi_{GL}(\cdot,\cdot,t)\|_{D(\Lambda)}
 \leq C_2 \varepsilon^{3/2}.
 $$
\end{theorem}
\begin{remark}{\rm
Since $ D(\Lambda) \subset H^2_{l,u}(\R\times (0,\pi),\R^3) $,
Sobolev's embedding theorem $ H^r_{l,u}(\R\times (0,\pi),\R^3) \subset C^0_{b,unif}(\R\times (0,\pi),\R^3)  $ for $ r \geq 2 $ gives 
the estimate
$$ 
\sup_{t \in [0,T_0/\varepsilon^2]} 
\sup_{(x,y) \in \R\times (0,\pi)}
\|(u,\Theta)(x,y,t)- \varepsilon \Psi_{GL}(x,y,t) \|_{\R^3} \leq C_2 \varepsilon^{3/2}.
$$}
\end{remark}

\section{Quasilinear reaction-diffusion systems}

\label{sec4}


The previous approach works for general quasilinear pattern forming
reaction-diffusion-advection systems, too.
We follow \cite{BS26} and reconsider 
\begin{equation}  \label{OS1} 
\partial_t U = D(U) \partial_x^2 U + f(U,\partial_x U)
\end{equation}
with $ t \geq 0 $, $ x \in \R $, $ U(x,t) \in \R^m $, 
$ D(U) = \textrm{diag}(d_1(U),\ldots,d_m(U)) $ with $ d_j:\R^m \to \R $ 
smooth and $ d_j(U) > 0 $ for all $ j = 1,\ldots,m $, and finally 
smooth $ f: \R^m \times \R^m \to \R^m $.

Similar to Section \ref{sec2} we assume 
\medskip

{\bf (ASS2)} There exists a $ t $- and $ x $-independent fixed point $ U^* $. The linearization 
$$
\partial_t V = \Lambda V = D(U^*) \partial_x^2 V + \partial_1 f(U^*,0) V + 
\partial_2 f(U^*,0) \partial_x V
$$ 
at $ U^* $ shows a Turing instability at a wave number $ k_c > 0 $
with a single curve of unstable eigenvalues.  

We refer to \cite[\S 9, \S 10]{SU17book} for the derivation of a Ginzburg-Landau equation,
\begin{equation} \label{GLRD}
\partial_T A =\nu_2 \partial_X^2 A+\nu_0 A - \nu_3 A|A|^2,
\end{equation}
with coefficients $ \nu_j \in \R $, 
with associated Ginzburg-Landau approximation $ \varepsilon \Psi_{GL} $, and estimates 
for the residual terms in $ H^r_{l,u} $-spaces.
We set 
$$ 
\cX = (H^r_{l,u})^m \qquad \textrm{and} \qquad D(\Lambda) =  (H^{r+2}_{l,u})^m.
$$
and have exactly  as above 
\begin{theorem}
Consider \eqref{OS1} and assume that the assumption 
{\bf (ASS2)} is fulfilled.

Let  $ \alpha \in (0,1) $ and let $ A \in C([0,T_0],H^{r+3+2\alpha}_{l,u}) $ be a solution of the 
Ginzburg-Landau equation \eqref{GLRD} satisfying 
\begin{equation} \label{eq19s2ben}
\sup_{T \in [0,T_0/\varepsilon^2]} \| A(\cdot,T) \|_{H^{r+3+2\alpha}_{l,u}}
\leq C_{GL}.
\end{equation}
Then  there exist $ \varepsilon_0 > 0 $ 
and $ C_2 > 0 $, only depending on 
$ C_{GL} > 0 $, $ T_0 > 0 $, $ r \geq 3 $
 such that  for all
$ \varepsilon \in (0,\varepsilon_0) $  there are solutions 
$ U $ of the 
general quasilinear pattern forming
reaction-diffusion-advection system
\eqref{OS1}
such that 
$$ 
\sup_{t \in [0,T_0/\varepsilon^2]} 
\| U(\cdot,t) - (U^*+ \varepsilon \Psi_{GL}(\cdot,t)) \|_{D(\Lambda)} \leq C_2 \varepsilon^{3/2}.
$$
\end{theorem}

%
\begin{remark}{\rm
Sobolev's embedding theorem $ H^r_{l,u}(\R,\R^m) \subset C^0_{b,unif}(\R,\R^m)  $ for $ r \geq 1 $ gives 
the estimate
$$ 
\sup_{t \in [0,T_0/\varepsilon^2]} 
\sup_{x \in \R}
\| U(x,t) - (U^*+ \varepsilon \Psi_{GL}(x,t)) \|_{\R^m} \leq C_2 \varepsilon^{3/2}.
$$}
\end{remark}
\begin{remark}{\rm
The assumption  that the diffusion 
matrix $ D(U) $ is  diagonal 
can obviously be weakened to the assumption that 
$ \Lambda $ is sectorial in $ \cX = (H^r_{l,u})^m $. }
\end{remark}

\section{Discussion and Outlook}
\label{sec8}

In the previous section we developed an approximation theory for  quasilinear pattern-forming systems in function spaces containing functions which not necessarily vanish for $ |x| \to \infty $.
We  applied this approach to three relevant classes of pattern-forming systems,
namely scalar toy problems, hydrodynamical stability problems 
and reaction-diffusion systems.
We expect that the presented method applies to all systems where a Ginzburg-Landau 
approximation can formally be performed, 
the linearization $ \Lambda $ around the trivial solution is a sectorial operator in some space $ \cX $,
the nonlinear terms smoothly map $ D(\Lambda) $ in $ \cX $, and the critical and stable modes can be separated by mode filters.

Since the attractor of these systems contains  spatially periodic pattern or modulating front solutions, it is essential to handle  function spaces such as 
uniformly local Sobolev spaces 
$ H^r_{l,u} $ or H\"older spaces $ C^{m,\nu} $. Global existence results 
with uniform bounds can only be expected 
in such spaces and not in Sobolev spaves.

We strongly expect that with the presented approach, 
also the second corner stone of this theory, namely the attractivity,
can be transferred to quasilinear systems, too.

Finally, we expect that the presented Ginzburg-Landau theory for 
quasilinear systems can be transferred to all
other instabilities which have been handled in the semilinear case.
These are, for instance, Turing-Hopf instabilities, long wave Hopf instabilities,
Turing or Turing-Hopf 
instabilities in systems with a conservation law or  time-periodic pattern forming systems, see \cite[\S 10]{SU17book} for an overview.

\bibliographystyle{alpha}
\bibliography{GLbib}

\end{document}